\documentclass[12pt,a4paper]{amsart}
\usepackage{amsmath, amsthm, amscd, amsfonts,graphicx}
\usepackage[matrix,arrow]{xy}
 \newtheorem{proposition}{Proposition}[section]
\newtheorem{theorem}{Theorem}[section]
\newtheorem{lemma}{Lemma}[section]
\newtheorem{corollary}{Corollary}[section]

\newtheorem{example}{Example}[section]

\theoremstyle{definition}
\newtheorem{definition}{Definition}[section]

\newtheorem{remark}{Remark}[section]

\begin{document}

\title[Monotone-light factorization for 2-categories]{The monotone-light factorization for 2-categories via 2-preorders}

 \author[Jo\~{a}o J. Xarez]{Jo\~{a}o J. Xarez}

 %\address{Department of Mathematics, University of Aveiro, Portugal.}

 %\email{isabel.andrade@ua.pt}

 \address{CIDMA - Center for Research and Development in Mathematics and Applications,
Department of Mathematics, University of Aveiro, Portugal.}

 \email{xarez@ua.pt}

 \subjclass[2020]{18A32,18E50,18N10}

 \keywords{Monotone-light factorization, 2-categories}

%\dedicatory{This article is dedicated to ...}

 \begin{abstract}
It is shown that the reflection $2Cat\rightarrow 2Preord$ of the category of all 2-categories into the category of 2-preorders determines a monotone-light factorization system on $2Cat$ and that the light morphisms are precisely the 2-functors faithful on 2-cells with respect to the vertical structure. In order to achieve such result it was also proved that the reflection $2Cat\rightarrow 2Preord$ has stable units, a stronger condition than admissibility in categorical Galois theory, and that 2-functors surjective both on horizontally composable triples of vertically composable pairs and on vertically composable triples of horizontally composable pairs of 2-cells are effective descent morphisms in $2Cat$.
\end{abstract}

\maketitle

\section{Introduction}\label{sec-introduction}

\subsection{The process of stabilization and localization}\label{sec-intro1}

The classical monotone-light factorization is due to S. Eilenberg \cite{Eil:trans} and G. T. Whyburn \cite{Why:trans}. It consists of a pair $(\mathcal{E^{'}},\mathcal{M^{*}})$ of classes of morphisms in the category $CompHaus$ of compact Hausdorff spaces, such that any continuous map $f$ of compact Hausdorff spaces factorizes as $f=m\circ e$, with $e\in\mathcal{E^{'}}$ (the class of monotone maps, i.e., those who have totally disconnected fibres) and $m\in\mathcal{M^{*}}$ (the class of light maps, i.e., those whose fibres are connected).

This factorization system $(\mathcal{E^{'}},\mathcal{M^{*}})$ may be obtained from another (reflective) factorization system $(\mathcal{E},\mathcal{M})$ on $CompHaus$ by a process of simultaneous stabilization and localization. That is: beginning with the reflector $I:CompHaus\rightarrow Stone$ (the category of Stone spaces), whose right adjoint $H$ is a full inclusion, there is an associated (reflective) factorization system $(\mathcal{E},\mathcal{M})$; then, we take the largest subclass $\mathcal{E^{'}}$ of $\mathcal{E}$ which is stable under pullbacks, and take $\mathcal{M^{*}}$ to be the class of morphisms which are obtained by localizing $\mathcal{M}$ ($m\in\mathcal{M^{*}}$ if there is any surjective map $p$ such that the pullback $p^*(m)$ along $p$ is in $\mathcal{M}$).

The process explained in the last paragraph was studied in general in \cite{CJKP:stab}, beginning with a reflection $H\vdash I:\mathbb{C}\rightarrow \mathbb{X}$, and the surjective maps $p$ used in the localization were generalized to morphisms $p:E\rightarrow B$ such that the pullback functor $p^{*}:\mathbb{C}/B\rightarrow\mathbb{C}/E$ is monadic, called effective descent morphisms (e.d.m.)\ in Grothendieck (basic-fibrational) descent theory.

In that paper \cite{CJKP:stab} some examples were given in which this process actually achieves a new factorization system $(\mathcal{E^{'}},\mathcal{M^{*}})$ (non-trivial because $\mathcal{E^{'}}\neq\mathcal{E}\Leftrightarrow\mathcal{M}\neq\mathcal{M^{*}}$). This is a quite rare phenomenon, since the pair $(\mathcal{E^{'}},\mathcal{M^{*}})$ usually fails to be a factorization system. When $(\mathcal{E^{'}},\mathcal{M^{*}})$ is in fact a factorization system (whose definition can be found in \cite{def_fact_sys}, and nicely exposed with other insights in \cite{CJKP:stab}), obtained by simultaneous stabilization and localization, it is to be called monotone-light as in \cite{CJKP:stab}, paying tribute to that first example referred to above.

An interesting feature of this process is its connection to categorical Galois theory (see \cite{G. Janelidze}). If the reflection $I\dashv H$ is semi-left-exact (in the sense of \cite{CHK:fact}), also called admissible in categorical Galois theory, then $\mathcal{M^{*}}$ is the class of coverings in the sense of that theory (being $Spl(E,p)$ the full subcategory of $\mathbb{C}/B$, determined by the coverings over $B$ split by the monadic extension $(E,p)$, the fundamental theorem of categorical Galois theory says that $Spl(E,p)$ has an algebraic description).

\subsection{Past and present work}\label{sec-intro2}
In \cite{X:ml}, it was presented a new non-trivial example of the process above, for the reflection $Cat\rightarrow Preord$ of the category of all categories into the category of all preordered sets, where the coverings are the faithful functors.

Now, in this paper, it will be proved that also the reflection $2Cat\rightarrow 2Preord$ of the category of all 2-categories into the category of all 2-preorders has a non-trivial monotone-light factorization, where the coverings are the 2-functors which are faithful vertically with respect to 2-cells.

Notice that both reflections have stable units (in the sense of \cite{CHK:fact}; a stronger condition than semi-left-exactness), which is crucial to the proof in association with the fact that there are enough e.d.m.\ with domain in the subcategory.

The needed characterization of e.d.m.\ in $2Cat$ is given in this paper, obtained in a completely analogous way as the characterization of e.d.m.\ in $Cat$ was done in \cite{JST:edm}. These characterizations depend on the embedding of $Cat$ and $2Cat$ in the obvious presheaf categories.

\subsection{Future work: prospects}\label{sec-intro3}

This paper is intended to be a first step in showing that the monotone-light factorization in \cite{X:ml} can be extended to higher category theory.

In particular, we hope to achieve the same results for $n$-categories in general (and to $\omega$-categories) in a similar way, embedding $n$-categories in categories of presheaves, if feasible.

We also believe, for the moment, that the good context for extending our results to higher category theory will be that of $\mathcal{V}$-categories. Remark that, considering $\mathcal{V}=Set$ the category of sets and then iterating one obtains $n$-categories. Of course, this opens the possibility of doing so for $\mathcal{V}$ other than the category of sets. We would be very interested, for instance, to apply these future results to the following open problem: is there a monotone-light factorization for semigroups via semilattices (see \cite{Sgr->Slat} and \cite{IsabelPhD})?

Remark finally that, because of the characterization given in this paper for the e.d.m in $2Cat$ (cf.\ \ref{proposition:EDM(2Cat)}), we are driven to present the following conjecture about the nature of e.d.m.\ for $n$-categories ($n > 1$), which may be helpful in related future work.

\textbf{Conjecture}: an $n$-functor is an e.d.m.\ in the category of all $n$-categories if and only if it is surjective both on horizontally composable triples of vertically composable pairs and on vertically composable triples of horizontally composable pairs of k-cells, for every $k\in\{2,...,n\}$.

Similarly, we could present other obvious conjectures. Some concerning the classes of morphisms of categorical Galois theory characterized in this paper, and even of other important classes of morphisms not treated here, since we were not exhaustive.

\section{The category of all 2-categories}\label{sec:2Cat}

Consider the category $2Cat$, with objects all 2-categories and whose morphisms are the 2-functors (see \cite[\S XII.3]{SM:cat}). Its definition is going to be stated in a way that suits our purposes. In order to do so, some intermediate structures need to be defined first.

First, consider the category $\mathbb{P}$ generated by the following \textit{precategory diagram},\\
\begin{picture}(200,40)(0,0)
\put(60,0){$P_2$}
\put(80,20){\vector(1,0){40}}\put(80,5){\vector(1,0){40}}\put(80,-10){\vector(1,0){40}}
\put(95,-5){$r$}\put(95,8){$m$}\put(95,25){$q$}
\put(125,0){$P_1$}
\put(140,20){\vector(1,0){40}}\put(180,5){\vector(-1,0){40}}\put(140,-10){\vector(1,0){40}}
\put(155,-5){$c$}\put(155,8){$e$}\put(155,25){$d$}
\put(185,0){$P_0$}
\end{picture}\\

\noindent in which

$d\circ e=1_{P_0}=c\circ e,\ d\circ m =d\circ q,\ c\circ m =c\circ r\ and\ c\circ q = d\circ r,$

\noindent where $1_{P_0}$ stands for the identity morphism of $P_0$ (see \cite[\S 4.1]{CJKP:stab}).

A \textit{precategory} is an object in the category of presheaves $\hat{\mathbb{P}}=Set^\mathbb{P}$, that is, any functor $P:\mathbb{P}\rightarrow Set$ to the category of sets.

If \begin{picture}(100,40)(0,0)
\put(60,0){$Q_2$}
\put(80,20){\vector(1,0){40}}\put(80,5){\vector(1,0){40}}\put(80,-10){\vector(1,0){40}}
\put(95,-5){$r'$}\put(95,8){$m'$}\put(95,25){$q'$}
\put(125,0){$Q_1$}
\put(140,20){\vector(1,0){40}}\put(180,5){\vector(-1,0){40}}\put(140,-10){\vector(1,0){40}}
\put(155,-5){$c'$}\put(155,8){$e'$}\put(155,25){$d'$}
\put(185,0){$Q_0$}
\end{picture}\\

\noindent is another precategory diagram, then a triple $(f_2,f_1,f_0)$ with $f_2:P_2\rightarrow Q_2$, $f_1:P_1\rightarrow Q_1$ and $f_0:P_0\rightarrow Q_0$, will be called a \textit{precategory morphism diagram} provided the following equations hold: $f_0\circ d=d'\circ f_1,\ f_0\circ c=c'\circ f_1,\ f_1\circ e=e'\circ f_0,\ f_1\circ q=q'\circ f_2, f_1\circ m=m'\circ f_2,\ f_1\circ r=r'\circ f_2.$\\

Secondly, consider the category $2\mathbb{P}$ generated by the following \textit{2-precategory diagram},

\begin{picture}(200,120)

\put(0,70){$hvP_2$}
\put(55,90){\vector(1,0){70}}\put(55,75){\vector(1,0){70}}\put(55,60){\vector(1,0){70}}
\put(70,65){$hr\times hr$}\put(70,78){$hm\times hm$}\put(70,95){$hq\times hq$}
\put(148,70){$vP_2$}
\put(185,90){\vector(1,0){70}}\put(255,75){\vector(-1,0){70}}\put(185,60){\vector(1,0){70}}
\put(200,65){$hc\times hc$}\put(200,78){$he\times he$}\put(200,95){$hd\times hd$}
\put(260,70){$P_0$}

\put (-55,35){$vr\times vr$}\put (-15,60){\vector(0,-1){40}}\put (-12,35){$vm\times vm$}\put (15,60){\vector(0,-1){40}}\put (42,35){$vq\times vq$}\put (40,60){\vector(0,-1){40}}

\put (120,35){$vr$}\put (135,60){\vector(0,-1){40}}\put (140,35){$vm$}\put (160,60){\vector(0,-1){40}}\put (165,35){$vq$}\put (180,60){\vector(0,-1){40}}

\put (270,35){$1_{P_0}$}\put (265,60){\vector(0,-1){40}}

\put(0,0){$hP_2$}
\put(55,20){\vector(1,0){70}}\put(55,5){\vector(1,0){70}}\put(55,-10){\vector(1,0){70}}
\put(80,-5){$hr$}\put(80,8){$hm$}\put(80,25){$hq$}
\put(148,0){$2P_1$}
\put(185,20){\vector(1,0){70}}\put(255,5){\vector(-1,0){70}}\put(185,-10){\vector(1,0){70}}
\put(210,-5){$hc$}\put(210,8){$he$}\put(210,25){$hd$}
\put(260,0){$P_0$}\put(300,0){$(2.1)$}

\put (-55,-35){$vc\times vc$}\put (-15,-10){\vector(0,-1){40}}
\put (-9,-35){$ve\times ve$}\put (15,-50){\vector(0,1){40}}
\put (42,-35){$vd\times vd$}\put (40,-10){\vector(0,-1){40}}

\put (120,-35){$vc$}\put (135,-10){\vector(0,-1){40}}
\put (140,-35){$ve$}\put (160,-50){\vector(0,1){40}}
\put (165,-35){$vd$}\put (180,-10){\vector(0,-1){40}}

\put (270,-35){$1_{P_0}$}\put (265,-10){\vector(0,-1){40}}

\put(0,-70){$P_2$}
\put(55,-50){\vector(1,0){70}}\put(55,-65){\vector(1,0){70}}\put(55,-80){\vector(1,0){70}}
\put(80,-75){$r$}\put(80,-62){$m$}\put(80,-45){$q$}
\put(148,-70){$P_1$}
\put(185,-50){\vector(1,0){70}}\put(255,-65){\vector(-1,0){70}}\put(185,-80){\vector(1,0){70}}
\put(210,-75){$c$}\put(210,-62){$e$}\put(210,-45){$d$}
\put(260,-70){$P_0$\hspace{10pt,}}

\end{picture}\\\\\\\\\\\\\\

\noindent in which:
\begin{itemize}
\item each one of the three horizontal diagrams (upwards, $P$, $hP$ and $hvP$) is a precategory diagram;

\item each one of the three vertical diagrams (from the left to the right, $vhP$, $vP$ and the trivial $P_0$) is a precategory diagram;

\item $(vc\times vc,vc,1_{P_0})$, $(ve\times ve,ve,1_{P_0})$, $(vd\times vd,vd,1_{P_0})$, $(vr\times vr,vr,1_{P_0})$, $(vm\times vm,vm,1_{P_0})$, $(vq\times vq,vq,1_{P_0})$ are all six precategory morphism diagrams (equivalently, $(hq\times hq,hq,q)$, $(hm\times hm,hm,m)$, $(hr\times hr,hr,r)$, $(hd\times hd,hd,d)$, $(he\times he,he,e)$, $(hc\times hc,hc,c)$ are all six precategory morphism diagrams).
\end{itemize}

Notice that the names given to objects and morphisms in $(2.1)$ are arbitrary, being so chosen in order to relate to the following last definition of section \ref{sec:2Cat} (for instance, $vq\times vq$ will denote the morphism uniquely determined by a pullback diagram).\\

The category $2Cat$ of all 2-categories is the full subcategory of $\hat{2\mathbb{P}}=Set^{2\mathbb{P}}$, determined by its objects $C:2\mathbb{P}\rightarrow Set$ such that the image by $C$ of each horizontal and vertical precategory diagram in $(2.1)$ is a category. That is, for instance, in the case of the bottom horizontal precategory diagram in $(2.1)$:

the commutative square

\begin{picture}(100,80)
\put (30,0){$C(P_1)$}\put (30,50){$C(P_2)$}\put (135,0){$C(P_0)$}\put (135,50){$C(P_1)$}
\put (30,25){$Cr$}\put (90,60){$Cq$}
\put(160,25){$Cc$}\put(200,25){$(2.2)$}
\put(90,10){$Cd$}
\put(70,3){\vector(1,0){58}}\put(70,55){\vector(1,0){58}}
\put(50,45){\vector(0,-1){33}}\put (155,45){\vector(0,-1){33}}
\end{picture}\\

\noindent is a pullback diagram in $Set$;

the associative and unit laws hold for the operation $Cm$, that is, the following respective diagrams commute in $Set$,

\begin{picture}(100,80)
\put (30,0){$C(P_2)$}\put (0,50){$C(P_2)\times_{C(P_1)}C(P_2)$}\put (165,0){$C(P_1)$,}\put (165,50){$C(P_2)$}
\put (-5,25){$Cr\times Cm$}\put (100,60){$Cm\times Cq$}
\put(190,25){$Cm$\hspace{20pt}$(2.3)$}
\put(100,10){$Cm$}
\put(70,3){\vector(1,0){88}}\put(103,55){\vector(1,0){55}}
\put(50,45){\vector(0,-1){33}}\put (185,45){\vector(0,-1){33}}
\end{picture}\\

\begin{picture}(300,80)

\put (-60,50){$C(P_0)\times_{C(P_0)}C(P_1)$}\put (20,45){\vector(0,-1){30}}\put(42,55){\vector(1,0){70}}
\put(230,55){\vector(-1,0){70}}

\put(0,0){$C(P_1)$}\put (25,25){$pr_2$}
\put(42,3){\vector(1,0){70}}
\put(50,60){$Ce\times1_{C(P_1)}$}\put(60,10){$1_{C(P_1)}$}

\put (120,0){$C(P_1)$}\put (120,50){$C(P_2)$}\put (245,0){$C(P_1)$\hspace{10pt.}}\put (235,50){$C(P_1)\times_{C(P_0)}C(P_0)$}

\put (145,25){$Cm$}\put (170,60){$1_{C(P_1)}\times Ce$}\put(270,25){$pr_1$\hspace{20pt}$(2.4)$}
\put(190,10){$1_{C(P_1)}$}

\put(230,3){\vector(-1,0){70}}
\put(140,45){\vector(0,-1){30}}\put (265,45){\vector(0,-1){30}}
\end{picture}\\

It would be a long and trivial calculation to check that there is an isomorphism between the category of all 2-categories (in the sense of \cite[\S XII.3]{SM:cat}) and the full subcategory of $\hat{2\mathbb{P}}$ just defined. Notice that: the requirement that the horizontal composite of two vertical identities is itself a vertical identity is encoded in diagram $(2.1)$ in the commutativity of the square $hm\circ (ve\times ve)=ve\circ m$ ; the interchange law, which relates the vertical and the horizontal composites of 2-cells, is encoded in diagram $(2.1)$ in the commutativity of the square $vm\circ (hm\times hm)=hm\circ (vm\times vm)$.

\section{Internal categories and limits}\label{sec:Internal categories and limits}
In section \ref{sec:2Cat}, if the category $Set$ of sets is replaced by any category $\mathcal{C}$ with pullbacks, then one obtains the definition of $2Cat(\mathcal{C})$, the category of internal 2-categories in $\mathcal{C}$.

In this section \ref{sec:Internal categories and limits}, the goal is to show that the category of all 2-categories $2Cat$ is closed under limits in the presheaves category $\hat{2\mathbb{P}}=Set^{2\mathbb{P}}$. The following Lemmas \ref{lemma:pullbacks internal cats} and \ref{lemma:products internal cats} give some well known facts about limits of internal categories, which will translate into internal 2-categories, and finally into 2-categories in the special case of $\mathcal{C}=Set$.

In what follows, $Cat(\mathcal{C})$ will denote the category of internal categories in $\mathcal{C}$, that is, the full subcategory of the category of functors $\mathcal{C}^\mathbb{P}$, determined by all the functors $C:\mathbb{P}\rightarrow \mathcal{C}$ such that the diagram $(2.2)$ is a pullback diagram in $\mathcal{C}$ and the diagrams $(2.3)$ and $(2.4)$ commute in $\mathcal{C}$ ($\mathbb{P}$ is of course the category defined in section \ref{sec:2Cat}).

\begin{lemma}\label{lemma:pullbacks internal cats}
Let $\mathcal{C}$ be a category with pullbacks.

Then, $Cat(\mathcal{C})$ is closed under pullbacks in $\mathcal{C}^\mathbb{P}$, where pullbacks exist and are calculated pointwise.
\end{lemma}

\begin{lemma}\label{lemma:products internal cats}
Let $\mathcal{C}$ be a category with pullbacks.

If $\mathbb{I}$ is a discrete category (that is, a set) and $\mathcal{C}$ has all limits $\mathbb{I}\rightarrow \mathcal{C}$, then $Cat(\mathcal{C})$ is closed under all limits $\mathbb{I}\rightarrow Cat(\mathcal{C})$ in $\mathcal{C}^\mathbb{P}$, where limits $\mathbb{I}\rightarrow \mathcal{C}^\mathbb{P}$ exist and are calculated pointwise.
\end{lemma}

\begin{corollary}\label{corollary:limits 2Cat}

If $\mathcal{C}$ has all limits then $2Cat(\mathcal{C})$ is closed under limits in the functor category $\mathcal{C}^{2\mathbb{P}}$, where all limits exist and are calculated pointwise.

In particular, for $\mathcal{C}=Set$, $2Cat$ is closed under limits in $\hat{2\mathbb{P}}=Set^{2\mathbb{P}}$.
\end{corollary}

\begin{proof}

The proof follows from the fact that limits are calculated pointwise in $\mathcal{C}^{2\mathbb{P}}$, and that a category with pullbacks and all products has all limits, and from Lemmas \ref{lemma:pullbacks internal cats} and \ref{lemma:products internal cats}.
\end{proof}

\section{Effective descent morphisms in 2Cat}
\label{sec:edm in 2Cat}

Consider again the category of all categories $Cat$ and its full inclusion in the category of precategories $\hat{\mathbb{P}}=Set^\mathbb{P}$.

A functor $p:\mathbb{E}\rightarrow \mathbb{B}$ is an effective descent morphism (e.d.m.)\footnote{Also called a \textit{monadic extension} in categorical Galois theory.} in $Cat$ if and only if it is surjective on composable triples of morphisms. The proof of this statement can be found in \cite[Proposition 6.2]{JST:edm}. In a completely analogous way, a class of effective descent morphisms in $2Cat$ is going to be given in the following Proposition \ref{proposition:EDM(2Cat)}.

\begin{proposition}\label{proposition:EDM(2Cat)}

 A 2-functor $2p:2\mathbb{E}\rightarrow 2\mathbb{B}$ is an e.d.m.\ in the category of all 2-categories $2Cat$ if it is surjective both on
\begin{itemize}
\item vertically composable triples of horizontally composable pairs of 2-cells, and on

\item horizontally composable triples of vertically composable pairs of 2-cells.
\end{itemize}
\end{proposition}

\begin{proof}\

Please confer the following Example \ref{example:EDM(2Cat)}, for the exact meaning of the statement.

%\noindent (if:)

Let $2p:2\mathbb{E}\rightarrow 2\mathbb{B}$ be surjective on vertically/horizontally composable triples of horizontally/vertically composable pairs of 2-cells. Then, $2p$ is an e.d.m.\ in $\hat{2\mathbb{P}}=Set^{2\mathbb{P}}$, since the effective descent morphisms in a category of presheaves are simply those surjective pointwise (which, of course, is implied by either surjectivity on triples of composable 2-cells). Hence, the following instance of \cite[Corollary 3.9]{JST:edm} can be applied:

if $2p:2\mathbb{E}\rightarrow 2\mathbb{B}$ in $2Cat$ is an e.d.m.\ in $\hat{2\mathbb{P}}=Set^{2\mathbb{P}}$ then $2p$ is an e.d.m.\ in $2Cat$ if and only if, for every pullback square

\begin{picture}(100,80)
\put (40,0){$2\mathbb{E}$}\put (40,50){$2\mathbb{D}$}\put (145,0){$2\mathbb{B}$}\put (145,50){$2A$}
%\put (30,25){$Cr$}\put (90,60){$Cq$}\put(160,25){$Cc$}
\put(200,25){$(4.1)$}
\put(90,10){$2p$}
\put(70,3){\vector(1,0){58}}\put(70,55){\vector(1,0){58}}
\put(50,45){\vector(0,-1){33}}\put (155,45){\vector(0,-1){33}}
\end{picture}\\

\noindent in $\hat{2\mathbb{P}}=Set^{2\mathbb{P}}$ such that $2\mathbb{D}$ is in $2Cat$, then also $2A$ is in $2Cat$.

Since the pullback square $(4.1)$ is calculated pointwise (cf. Corollary \ref{corollary:limits 2Cat}), it induces six other pullback squares in $\hat{\mathbb{P}}=Set^{\mathbb{P}}$, corresponding to the three rows $P$, $hP$ and $hvP$, and the three columns $vhP$, $vP$ and $P_0$, in the 2-precategory diagram $(2.1)$.

The fact that $2p$ is surjective on vertically/horizontally composable triples of horizontally/vertically composable pairs of 2-cells, implies that its six restrictions (to the six rows and columns $2\mathbb{E}(P)$, $2\mathbb{E}(hP)$, $2\mathbb{E}(hvP)$, $2\mathbb{E}(vhP)$, $2\mathbb{E}(vP)$ and $2\mathbb{E}(P_0)$) are surjective on triples of composable morphisms in $Cat$, as it is easy to check. Hence, these six restrictions are effective descent morphisms in $Cat$. Therefore, $2A$ must always be a 2-category, provided so is $2\mathbb{D}$.%\\

\end{proof}

\begin{example}\label{example:EDM(2Cat)}

It is obvious that the coproduct of 2-categories is just the disjoint union, as for categories.

Let $vh\mathbf{4}$ and $hv\mathbf{4}$ be the 2-categories generated by the following two diagrams, respectively:

\begin{picture}(60,40)

%vh4:

\put (0,0){$0$}\put (50,0){$1$}\put (100,0){$2$}

\put(120,0){;}

\put(10,30){\vector(1,0){35}}\put(10,12){\vector(1,0){35}}
\put(10,-6){\vector(1,0){35}}\put(10,-24){\vector(1,0){35}}

\put(23,18){$\Downarrow$}\put(23,0){$\Downarrow$}
\put(23,-18){$\Downarrow$}

\put(60,30){\vector(1,0){35}}\put(60,12){\vector(1,0){35}}
\put(60,-6){\vector(1,0){35}}\put(60,-24){\vector(1,0){35}}

\put(73,18){$\Downarrow$}\put(73,0){$\Downarrow$}
\put(73,-18){$\Downarrow$}

%hv4:
\put (150,0){$0$}\put (200,0){$1$}
\put (250,0){$2$}\put (300,0){$3$}

\put(320,0){.}

\put(160,18){\vector(1,0){35}}\put(160,0){\vector(1,0){35}}\put(160,-18){\vector(1,0){35}}
\put(210,18){\vector(1,0){35}}\put(210,0){\vector(1,0){35}}\put(210,-18){\vector(1,0){35}}
\put(260,18){\vector(1,0){35}}\put(260,0){\vector(1,0){35}}\put(260,-18){\vector(1,0){35}}

\put(173,6){$\Downarrow$}\put(223,6){$\Downarrow$}\put(273,6){$\Downarrow$}
\put(173,-12){$\Downarrow$}\put(223,-12){$\Downarrow$}\put(273,-12){$\Downarrow$}

\end{picture}\\\\\\

Consider, for each 2-category $2\mathbb{B}$, the 2-category $$2\mathbb{E}=(\coprod_{i\in I}vh\mathbf{4})+(\coprod_{j\in J}hv\mathbf{4}),$$ such that $I$ is the set of all vertically composable triples of horizontally composable pairs of 2-cells in $2\mathbb{B}$, and $J$ is the set of all horizontally composable triples of vertically composable pairs of 2-cells in $2\mathbb{B}$.

Then, there is an e.d.m.\ $2p:2\mathbb{E}\rightarrow 2\mathbb{B}$ which projects the corresponding copy of $vh\mathbf{4}$ and $hv\mathbf{4}$ to every $i\in I$ and every $j\in J$, respectively.\\

As another option, let $$2\mathbb{E}=\coprod_{k\in I\cup J}hvh\mathbf{4},$$ with $hvh\mathbf{4}$ the 2-category generated by the following diagram,

$$\begin{picture}(60,40)

%vh41:

\put (0,0){$0$}\put (50,0){$1$}

\put(10,30){\vector(1,0){35}}\put(10,12){\vector(1,0){35}}
\put(10,-6){\vector(1,0){35}}\put(10,-24){\vector(1,0){35}}

\put(23,18){$\Downarrow$}\put(23,0){$\Downarrow$}
\put(23,-18){$\Downarrow$}

%vh42:

\put (100,0){$2$}

\put(60,30){\vector(1,0){35}}\put(60,12){\vector(1,0){35}}
\put(60,-6){\vector(1,0){35}}\put(60,-24){\vector(1,0){35}}

\put(73,18){$\Downarrow$}\put(73,0){$\Downarrow$}
\put(73,-18){$\Downarrow$}

%vh43:

\put (150,0){$3$\hspace{10pt}.}

\put(110,30){\vector(1,0){35}}\put(110,12){\vector(1,0){35}}
\put(110,-6){\vector(1,0){35}}\put(110,-24){\vector(1,0){35}}

\put(123,18){$\Downarrow$}\put(123,0){$\Downarrow$}
\put(123,-18){$\Downarrow$}

\end{picture}$$\\

\end{example}

Let us now specify, in the following remark, what do we mean when saying a diagram generates a 2-category.

\begin{remark}\label{remark:EDM(2Cat)}
The 2-categories $vh\mathbf{4}$, $hv\mathbf{4}$ and $hvh\mathbf{4}$ are really 2-preorders, as defined just below at the beginning of Section \ref{sec:stable units}.

In fact, there is a free 2-preorder generated by a 2-relation (that is, a 2-graph plus a relation between 1-cells with the same initial and terminal 0-cell), built as follows:

starting with a 2-relation (as, for instance, the three diagrams in example \ref{example:EDM(2Cat)}), first one adds the $1Cat=Cat$ structure to the 0 and 1-cells (using the well known adjunction from graphs into categories, $Set^{\mathbb{G}}\rightarrow Cat$, with $\mathbb{G}$ the subcategory of $\mathbb{P}$ determined by $P_0, P_1$, $d$ and $c$; cf.\ Theorem 1 in \cite{SM:cat}[II.7]);

then, the intersection of all the relations on the finite strings of 1-cells, which constitute a 2-preorder and which include the original relation on the 1-cells (now the strings with just one arrow), gives the preorder structure on 2-cells (notice that there exists such a structure: relate every two strings with the same initial and terminal 0-cells);

hence, we have described an adjunction from $2Preord$ into $2Rel$ where the right adjoint is the forgetful functor ($2Rel$ is of course the full subcategory of the presheaves category $Set^{2\mathbb{G}}$, determined by the presheaves $G$ such that $G(vc)$ and $G(vd)$ are jointly monic; with $2\mathbb{G}$ the subcategory of $2\mathbb{P}$ determined by $2P1$, $P_1, P_0$, $vd$, $vc$, $d$ and $c$).

\end{remark}

\section{The reflection of 2-categories into 2-preorders has stable units and a monotone-light factorization}
\label{sec:stable units}

Let $2Preord$ be the full subcategory of $2Cat$ determined by the objects $C:2\mathbb{P}\rightarrow Set$ such that $Cvd$ and $Cvc$ are jointly monic (cf. diagram $(2.1)$), that is,

\begin{picture}(200,40)(0,0)
\put(40,0){$C(vP_2)$}
\put(80,20){\vector(1,0){40}}\put(80,5){\vector(1,0){40}}\put(80,-10){\vector(1,0){40}}
\put(90,-5){$Cvr$}\put(90,8){$Cvm$}\put(90,25){$Cvq$}
\put(125,0){$C(2P_1)$}
\put(165,20){\vector(1,0){40}}\put(205,5){\vector(-1,0){40}}\put(165,-10){\vector(1,0){40}}
\put(175,-5){$Cvc$}\put(175,8){$Cve$}\put(175,25){$Cvd$}
\put(210,0){$C(P_1)$} \put(300,0){$(5.1)$}
\end{picture}\\\\

\noindent is a preordered set.\\

There is a reflection

\begin{picture}(200,20)(0,0)

\put(0,0){$H\vdash I: 2Cat$}
\put(70,3){\vector(1,0){20}}
\put(95,0){$2Preord$,}
\put(310,0){$(5.2)$}

\put (150,0){$a$}\put (200,0){$b$}
\put (230,0){$a$}\put (280,0){$b$,}

\put(160,12){\vector(1,0){35}}\put(175,15){$f$}
\put(160,-6){\vector(1,0){35}}\put(175,-15){$g$}

\put(213,0){$\mapsto$}

\put(240,12){\vector(1,0){35}}\put(255,15){$f$}
\put(240,-6){\vector(1,0){35}}\put(255,-15){$g$}

\put(173,0){$\Downarrow$}\put(182,0){$\theta$}
\put(253,0){$\Downarrow$}\put(262,0){$\leq$}

\end{picture}\\

\noindent which identifies all 2-cells which have the same domain and codomain for the vertical composition. That is, the reflector $I$ takes the middle vertical category $C(vP)$ (cf. diagram $(5.1)$) to its image by the well known reflection $Cat\rightarrow Preord$ from categories into preordered sets (see \cite{X:ml}).\\

Many of the results in \cite{X:GenConComp} are going to be stated again, with small improvements in their presentation\footnote{The reader could easily bring these small improvements to the paper \cite{X:GenConComp}. In fact, although they are stated here in the particular case of the reflection from $2Cat$ into $2Preord$, they are completely general.}, in order to prove that the reflection $H\vdash I:2Cat\rightarrow 2Preord$ has stable units (in the sense of \cite{CHK:fact}).

\subsection{Ground structure} Consider the adjunction $H\vdash I:2Cat\rightarrow 2Preord$, described just above in $(5.2)$, with unit $\eta : 1_{2Cat}\rightarrow HI$.

\begin{itemize}

\item $2Cat$ has pullbacks (in fact, it has all limits - see Corollary \ref{corollary:limits 2Cat}).

\item $H$ is a full inclusion of $2Preord$ in $2Cat$, that is, $I$ is a reflection of a category with pullbacks into a full subcategory.

\item Consider also the forgetful functor $U:2Cat\rightarrow 2RGrph$, where $2RGrph$ is the presheaves category $Set^{2\mathbb{G}}$, with $2\mathbb{G}$ the category generated by the following \textit{2-reflexive graph diagram},\\

\begin{picture}(200,30)

\put(28,0){$2P_1$}
\put(65,20){\vector(1,0){70}}\put(135,5){\vector(-1,0){70}}\put(65,-10){\vector(1,0){70}}
\put(90,-5){$hc$}\put(90,8){$he$}\put(90,25){$hd$}
\put(140,0){$P_0$}

\put (0,-35){$vc$}\put (15,-10){\vector(0,-1){40}}
\put (20,-35){$ve$}\put (40,-50){\vector(0,1){40}}
\put (45,-35){$vd$}\put (60,-10){\vector(0,-1){40}}

\put (150,-35){$1_{P_0}$}\put (145,-10){\vector(0,-1){40}}

\put(28,-70){$P_1$}
\put(65,-50){\vector(1,0){70}}\put(135,-65){\vector(-1,0){70}}\put(65,-80){\vector(1,0){70}}
\put(90,-75){$c$}\put(90,-62){$e$}\put(90,-45){$d$}
\put(140,-70){$P_0$\hspace{10pt,}}

\end{picture}\\\\\\\\\\\\

\noindent satisfying the same equations as in the 2-precategory diagram $(2.1)$.

\item $\mathcal{E}$ denotes the class of all morphisms $(2g_1,g_1,g_0):G\rightarrow H$ of $2RGrph$ which are bijections on objects and on arrows, and surjections on 2-cells (that is, $g_0:G(P_0)\rightarrow H(P_0)$ and $g_1:G(P_1)\rightarrow H(P_1)$ are bijections, and $2g_1:G(2P_1)\rightarrow H(2P_1)$ is a surjection).

\item $\mathcal{T}=\{T\}$ is a singular set, with $T$ the 2-preorder generated by the diagram \hspace{10pt}\begin{picture}(95,30)

\put (0,0){$a$}\put (50,0){$a'$}

\put(70,0){$(5.3)$}

\put(10,12){\vector(1,0){35}}\put(23,17){$h$}
\put(10,-6){\vector(1,0){35}}\put(23,-17){$h'$}

\put(23,0){$\Downarrow$}\put(32,0){$\leq$}

\end{picture},\\\\

\noindent that is, a 2-preorder with two objects, two non-identity arrows and only one non-identity (both horizontally and vertically) 2-cell.\\

\end{itemize}

Then, the following four conditions are satisfied.\\

\begin{enumerate}

\item[(a)] $U$ preserves pullbacks (in fact, it preserves all limits).\\

\item[(b)] $\mathcal{E}$ is pullback stable in $2RGrph$, and if $g'\circ g$ is in $\mathcal{E}$ so is $g'$, provided $g$ is in $\mathcal{E}$.\footnote{In \cite{X:GenConComp}, it was also demanded in (b) that $\mathcal{E}$ is closed under composition, which is not needed. We take this opportunity to correct that redundancy in \cite{X:GenConComp}.}\\

\item[(c)] Every map $U\eta_C:U(C)\rightarrow UHI(C)$ belongs to $\mathcal{E}$, $C\in 2Cat$ (this is also obvious).\\

\item[(d)]\footnote{This item is rephrased from \cite{X:GenConComp}, in a way that seems to us now more easily understandable. Remark also that the diagram $(5.4)$ is simplified, suppressing one morphism $UH(T)\rightarrow UH(T)$, which can be the identity. We take this opportunity to correct that other redundancy.}Let $g:N\rightarrow M$ be any morphism of $2Preord$ such that $UHg:UH(N)\rightarrow UH(M)$ is in $\mathcal{E}$.

\noindent If,

\hspace{5pt} there is one morphism $f:A\rightarrow UH(N)$ of $2RGrph$ in $\mathcal{E}$

\hspace{5pt} such that,

\hspace{20pt} for all morphisms $c:T\rightarrow M$ in $2Preord$

\hspace{20pt} ($T$ as defined in (5.3)),

\hspace{20pt} there is a commutative diagram as below

%\begin{equation}\label{eq:GenConComp}
%\vcenter{
$$
\begin{picture}(100,60)
\put(0,0){$A$}\put(-30,50){$A\times_{UH(M)}UH(T)$}\put(80,0){UH(N)}\put(165,0){$UH(M)$}\put(165,50){$UH(T)$}
\put(-10,25){$pr_1$}\put(180,25){$UHc$}\put(250,25){$(5.4)$}

\put(40,-8){$f$}\put(125,-8){$UHg$}
\put(15,3){\vector(1,0){60}}\put(120,3){\vector(1,0){40}}\put(65,53){\vector(1,0){95}}\put(90,57){$pr_2$}
\put(7,45){\vector(0,-1){33}}\put(175,45){\vector(0,-1){33}}

\put(160,45){\vector(-1,-1){37}}

\end{picture}\\
$$
%}
%\end{equation}\\

\noindent then

\hspace{10pt} $g:N\rightarrow M$ is an isomorphism in $2Preord$.\\

\end{enumerate}

It remains to show that the statement in (d) is true, which is trivial, since if $g:N\rightarrow M$ is in $\mathcal{E}$, seen as a morphism of $2RGrph$, then $g$ must be an isomorphism in $2Preord$ by the uniqueness of the 2-cells in $N$ and in $M$.

\subsection{Stable units}\label{subsection:stable units}

Using the fact that a \textit{ground structure} holds (which guarantees the validity of Theorems 2.1 and 2.2 in \cite{X:GenConComp}), it will be possible to show that $H\vdash I: 2Cat\rightarrow 2Preord$ is an admissible reflection in the sense of categorical Galois theory (cf. \cite{G. Janelidze}) or, equivalently, semi-left-exact in the sense of \cite{CHK:fact}. Furthermore, it will be shown, always using the results in \cite{X:GenConComp}, that the reflection $H\vdash I: 2Cat\rightarrow 2Preord$ satisfies the stronger condition of having stable units.

\begin{definition}\label{def:connected component}

Consider any morphism $\mu :T\rightarrow HI(C)$ from $T$ ($\in\mathcal{T}$; cf. $(5.3)$), for some $C\in 2Cat$.

The \textit{connected component} of the morphism $\mu$ is the pullback $C_\mu
=C\times_{HI(C)}T$ in the following pullback square

%\begin{equation}\label{eq:connected component}
$$
\begin{picture}(100,60)
\put(0,0){$C$}\put(0,40){$C_\mu$}\put(80,0){$HI(C)$\hspace{10pt,}}\put(90,40){$T$}
\put(-15,20){$\pi_1^\mu$}\put(100,20){$\mu$}\put(120,20){$(5.5)$}
\put(40,10){$\eta_C$}\put(40,50){$\pi_2^\mu$}
\put(15,3){\vector(1,0){60}}\put(20,43){\vector(1,0){54}}
\put(4,35){\vector(0,-1){23}}\put(92,35){\vector(0,-1){23}}
\end{picture}
$$
%\end{equation}

\noindent where $\eta_C$ is the unit morphism of $C$ in the reflection $H\vdash I: 2Cat\rightarrow 2Preord$, and $T$ is identified with $H(T)$.

\end{definition}

\begin{theorem}\label{theorem:semi-left-exactness}

The reflection $H\vdash I: 2Cat\rightarrow 2Preord$ is semi-left-exact.

\end{theorem}

\begin{proof}
According to Theorem 2.1 in \cite{X:GenConComp}, one has to show that $I\pi_2^\mu:I(C_\mu)\rightarrow I(T)$ is an isomorphism, for every connected component $C_\mu$.\\

If $\mu($\begin{picture}(60,20)(0,0)

\put (0,0){$a$}\put (50,0){$a'$}

\put(10,12){\vector(1,0){35}}\put(25,16){$h$}
\put(10,-6){\vector(1,0){35}}\put(25,-16){$h'$}

\put(23,0){$\Downarrow$}\put(32,0){$\leq$}

\end{picture}$)=\ $\begin{picture}(60,20)(0,0)

\put (0,0){$c$}\put (50,0){$c'$,}

\put(10,12){\vector(1,0){35}}\put(25,16){$k$}
\put(10,-6){\vector(1,0){35}}\put(25,-16){$k'$}

\put(23,0){$\Downarrow$}\put(32,0){$\leq$}

\end{picture} then,\\\\

 \noindent since $U\eta_C\in \mathcal{E}$ (identity on objects and morphisms, and surjection on 2-cells), the pullback $C_\mu$ is the 2-category generated by the diagram\\

 \begin{picture}(150,20)(0,0)

\put (0,0){$(c,a)$}\put (110,0){$(c',a')$}

\put(30,12){\vector(1,0){80}}\put(60,16){$(k,h)$}
\put(30,-6){\vector(1,0){80}}\put(60,-16){$(k',h')$}

\put(63,0){$\Downarrow$}\put(72,0){$(\theta_r,\leq)$}

\end{picture},\\\\

\noindent with $\theta_r\in Hom_{C(vP)}(k,k')=\{\theta_r\ |\ r\in R\}$, the set $R$ indexing all the 2-cells $\theta_r$ in $C$ with vertical domain $k:c\rightarrow c'$ and vertical codomain $k':c\rightarrow c'$.

Hence, $I(C_\mu)\cong T$.\end{proof}

\begin{theorem}\label{theorem:stable units}

The reflection $H\vdash I: 2Cat\rightarrow 2Preord$ has stable units.

\end{theorem}

\begin{proof}
According to Theorem 2.2 in \cite{X:GenConComp}, one has to show that $I(C_\mu\times_TD_\nu)\cong T$, for every pair of connected components $C_\mu$, $D_\nu$, where $C_\mu\times_TD_\nu$ is the pullback object in any pullback of the form
$$
\begin{picture}(100,60)
\put(0,0){$C_{\mu}$}\put(-18,40){$C_\mu\times_{T}
D_{\nu}$}\put(95,0){$T$\hspace{10pt},}\put(90,40){$D_{\nu}$}
\put(-8,20){$p_1$}\put(100,20){$\pi_2^\nu$}
\put(45,10){$\pi_2^\mu$}\put(45,50){$p_2$}
\put(15,3){\vector(1,0){70}}\put(33,43){\vector(1,0){50}}
\put(4,35){\vector(0,-1){23}}\put(95,35){\vector(0,-1){23}}
\end{picture}
$$

\noindent where $\pi_2^\mu$ and $\pi_2^\nu$ are the second projections in pullback diagrams of the form $(5.5)$.\\

According to the previous Theorem \ref{theorem:semi-left-exactness}, one can suppose (up to isomorphism) that $C_\mu =\ $\begin{picture}(60,40)(0,0)

\put (0,0){$c$}\put (50,0){$c'$}

\put(10,12){\vector(1,0){35}}\put(25,16){$k$}
\put(10,-6){\vector(1,0){35}}\put(25,-16){$k'$}

\put(23,0){$\Downarrow$}\put(32,0){$\theta_r$}

\end{picture}, $r\in R$, and $D_\nu =\ $\begin{picture}(60,40)(0,0)

\put (0,0){$d$}\put (50,0){$d'$}

\put(10,12){\vector(1,0){35}}\put(25,16){$l$}
\put(10,-6){\vector(1,0){35}}\put(25,-16){$l'$}

\put(23,0){$\Downarrow$}\put(32,0){$\delta_s$}

\end{picture},\\\\

 \noindent $s\in S$ (the identity morphisms and the identity 2-cells are not displayed); the sets $R$ and $S$ indexing respectively all the 2-cells $\theta_r$ in $C$ with vertical domain $k:c\rightarrow c'$ and vertical codomain $k':c\rightarrow c'$, and all the 2-cells $\delta_s$ in $D$ with vertical domain $l:d\rightarrow d'$ and vertical codomain $l':d\rightarrow d'$.

 Hence, $C_\mu\times_TD_\nu =\ $\begin{picture}(150,40)(0,0)

\put (0,0){$(c,d)$}\put (110,0){$(c',d')$}

\put(30,12){\vector(1,0){80}}\put(60,16){$(k,l)$}
\put(30,-6){\vector(1,0){80}}\put(60,-16){$(k',l')$}

\put(63,0){$\Downarrow$}\put(72,0){$(\theta_r,\delta_s)$}

\end{picture}, $(r,s)\in R\times S$, and so it is obvious that $I(C_\mu\times_TD_\nu ) \cong\ $\begin{picture}(60,40)(0,0)

\put (0,0){$a$}\put (50,0){$a'$.}

\put(10,12){\vector(1,0){35}}\put(25,16){$h$}
\put(10,-6){\vector(1,0){35}}\put(25,-16){$h'$}

\put(23,0){$\Downarrow$}\put(32,0){$\leq$}

\end{picture}\\
\end{proof}

\subsection{Monotone-light factorization for 2-categories via 2-preorders}\label{subsection:existence m-l fact}

\begin{theorem}\label{theorem:monotone-light fact}

The reflection $H\vdash I: 2Cat\rightarrow 2Preord$ does have a monotone-light factorization.

\end{theorem}

\begin{proof}

The statement is a consequence of the central result of \cite{CJKP:stab} (cf. Corollary 6.2 in \cite{X:iml}), because $H\vdash I$ has stable units (cf. Theorem \ref{theorem:stable units}) and for every $2\mathbb{B}\in 2Cat$ there is an e.d.m.\ $2p:2\mathbb{E}\rightarrow 2\mathbb{B}$ with $2\mathbb{E}\in 2Preord$ (cf.\ Example \ref{example:EDM(2Cat)}).

\end{proof}

In the following section \ref{sec:Vertical and stably-vertical}, it will be proved that the monotone-light factorization system is not trivial. That is, it does not coincide with the reflective factorization system associated to the reflection of $2Cat$ into $2Preord$.

\section{Vertical and stably-vertical 2-functors}
\label{sec:Vertical and stably-vertical}

In this section, it will be given a characterization of the class of vertical morphisms $\mathcal{E}_I$ in the reflective factorization system $(\mathcal{E}_I,\mathcal{M}_I)$, and of the class of the stably-vertical morphisms $\mathcal{E}'_I$ ($\subseteq\mathcal{E}_I$)\footnote{$\mathcal{E}'_I$ is the largest subclass of $\mathcal{E}_I$ stable under pullbacks. The terminologies ``vertical morphisms" and ``stably-vertical morphisms" were introduced in \cite{for:ins:sep:factorization}.} in the monotone-light factorization system $(\mathcal{E}'_I,\mathcal{M}^*_I)$, both associated to the reflection $2Cat\rightarrow 2Preord$. Then, since $\mathcal{E}'_I$ is a proper class of $\mathcal{E}_I$, one concludes that $(\mathcal{E}'_I,\mathcal{M}^*_I)$ is a non-trivial monotone-light factorization system.\\

Consider a 2-functor $f:A\rightarrow B$, which is obviously determined by the three functions $f_0:A(P_0)\rightarrow B(P_0)$, $f_1:A(P_1)\rightarrow B(P_1)$ and $2f_1:A(2P_1)\rightarrow B(2P_1)$ (cf.\ diagram $(2.1)$), so that we may make the identification $f=(2f_1,f_1,f_0)$.

\begin{proposition}\label{proposition:vertical morphisms}

A 2-functor $f=(2f_1,f_1,f_0):A\rightarrow B$ belongs to the class $\mathcal{E}_I$ of vertical 2-functors if and only if the following two conditions hold:
\begin{enumerate}
\item $f_0$ and $f_1$ are bijections;
\item for every two elements $h$ and $h'$ in $A(P_1)$, if $Hom_{B(vP)}(f_1h,f_1h')$ is nonempty then so is $Hom_{A(vP)}(h,h')$.
\end{enumerate}
\end{proposition}

\begin{proof}

The 2-functor $f=(2f_1,f_1,f_0)$ belongs to $\mathcal{E}_I$ if and only if $If$ is an isomorphism (cf.\ \cite[\S 3.1]{CJKP:stab}), that is, $If_0$, $If_1$, and $I2f_1$ are bijections. Since $If_0=f_0$ and $If_1=f_1$, the fact that $f\in \mathcal{E}_I$ implies and is implied by (1) and (2) is trivial.

\end{proof}

\begin{proposition}\label{proposition:stably-vertical}

A 2-functor $f=(2f_1,f_1,f_0):A\rightarrow B$ belongs to the class $\mathcal{E}'_I$ of stably-vertical 2-functors if and only if the following two conditions hold:

\begin{enumerate}

\item $f_0$ and $f_1$ are bijections;
\item for every two elements $h$ and $h'$ in $A(P_1)$, $f$ induces a surjection $Hom_{A(vP)}(h,h')\rightarrow Hom_{B(vP)}(f_1h,f_1h')$ ($f$ is a ``full functor on 2-cells").

\end{enumerate}

\end{proposition}

\begin{proof}

As every pullback $g^*(f)=\pi_1:C\times_BA\rightarrow C$ in $2Cat$ of $f$ along any 2-functor $g:C\rightarrow B$ is calculated pointwise, and $(2f_1,f_1):A(vP)\rightarrow B(vP)$ is a stably-vertical functor for the reflection $Cat\rightarrow Preord$, that is, $f_1$ is a bijection and $(2f_1,f_1)$ is a full functor (cf.\ Propositions 4.4 and 3.2 in \cite{X:ml}), then (1) and (2) imply that $g^*(f)$ belongs to $\mathcal{E}_I$ (cf.\ last Proposition \ref{proposition:vertical morphisms}).

Hence, $f\in \mathcal{E}'_I$ if (1) and (2) hold.\\

If $f\in \mathcal{E}'_I$, then $f\in \mathcal{E}_I$ ($\mathcal{E}'_I\subseteq \mathcal{E}_I$), and therefore (1) holds.

Suppose now that (2) does not hold, so that there is $\theta :f_1h\rightarrow f_1h'$ not in the image of $f$, and consider the 2-category $C$ generated by the diagram \begin{picture}(60,30)(0,0)

\put (0,0){$b$}\put (50,0){$b'$}

\put(10,12){\vector(1,0){35}}\put(15,16){$f_1h$}
\put(10,-6){\vector(1,0){35}}\put(15,-16){$f_1h'$}

\put(23,0){$\Downarrow$}\put(32,0){$\theta$}

\end{picture}, and let $g$ be the inclusion of $C$ in $B$. Then, $C\times_BA\cong$ \begin{picture}(60,50)(0,0)

\put (0,0){$b$}\put (50,0){$b'$}

\put(10,12){\vector(1,0){35}}\put(15,16){$f_1(h)$}
\put(10,-6){\vector(1,0){35}}\put(15,-16){$f_1(h')$}

\end{picture}, with no non-identity 2-cells, and so $g^*(f)$ is not in $\mathcal{E}_I$.\\\\

Hence, if $f\in \mathcal{E}'_I$ then (1) and (2) must hold.\end{proof}

It is evident that $\mathcal{E}'_I$ is a proper class of $\mathcal{E}_I$, therefore the monotone-light factorization system $(\mathcal{E}'_I,\mathcal{M}^*_I)$ is non-trivial ($\neq (\mathcal{E}_I,\mathcal{M}_I)$).

\section{Trivial coverings for 2-categories via 2-preorders}
\label{sec:Trivial coverings}

A 2-functor $f:A\rightarrow B$ belongs to the class $\mathcal{M}_I$ of trivial coverings (with respect to the reflection $H\vdash I:2Cat\rightarrow 2Preord$) if and only if the following square

\begin{picture}(100,80)
\put (45,0){$B$}\put (45,50){$A$}\put (140,0){$I(B)$}\put (140,50){$I(A)$}
\put (35,25){$f$}\put (90,60){$\eta_A$}
\put(160,25){$If$}\put(200,25){$(7.1)$}
\put(90,10){$\eta_B$}
\put(70,3){\vector(1,0){58}}\put(70,55){\vector(1,0){58}}
\put(50,45){\vector(0,-1){33}}\put (155,45){\vector(0,-1){33}}
\end{picture}\\

\noindent is a pullback diagram, where $\eta_A$ and $\eta_B$ are unit morphisms for the reflection $H\vdash I:2Cat\rightarrow 2Preord$ (cf.\ \cite[Theorem 4.1]{CHK:fact}).\\

Since the pullback (as any limit) is calculated pointwise in $2Cat$ (cf.\ Corollary \ref{corollary:limits 2Cat}), then $f\in \mathcal{M}_I$ if and only if the following seven squares are pullbacks, corresponding to the seven pointwise components of $\eta_A$ and of $\eta_B$ (cf.\ diagram $(2.1)$):

\begin{picture}(100,80)
\put (37,0){$B(P_i)$}\put (37,50){$A(P_i)$}\put (135,0){$I(B)(P_i)$}\put (135,50){$I(A)(P_i)$}
\put (33,25){$f_{P_i}$}\put (90,60){$\eta_{A(P_i)}$}
\put(160,25){$If_{P_i}$}\put(87,25){($D_i$)}\put(200,25){$(i=0,1,2);$}
\put(90,10){$\eta_{B(P_i)}$}
\put(70,3){\vector(1,0){58}}\put(70,55){\vector(1,0){58}}
\put(50,45){\vector(0,-1){33}}\put (155,45){\vector(0,-1){33}}
\end{picture}

\begin{picture}(400,80)
\put (-4,0){$B(2P_1)$}\put (-4,50){$A(2P_1)$}\put (100,0){$I(B)(2P_1)$}\put (100,50){$I(A)(2P_1)$}
\put (-7,25){$f_{2P_1}$}\put (55,60){$\eta_{A(2P_1)}$}
\put(125,25){$If_{2P_1}$;}\put(52,25){($2D$)}
\put(55,10){$\eta_{B(2P_1)}$}
\put(35,3){\vector(1,0){58}}\put(35,55){\vector(1,0){58}}
\put(15,45){\vector(0,-1){33}}\put (120,45){\vector(0,-1){33}}

\put (196,0){$B(vP_2)$}\put (196,50){$A(vP_2)$}\put (300,0){$I(B)(vP_2)$}\put (300,50){$I(A)(vP_2)$}
\put (193,25){$f_{vP_2}$}\put (255,60){$\eta_{A(vP_2)}$}
\put(325,25){$If_{vP_2}$;}\put(252,25){($vD$)}
\put(255,10){$\eta_{B(vP_2)}$}
\put(235,3){\vector(1,0){58}}\put(235,55){\vector(1,0){58}}
\put(215,45){\vector(0,-1){33}}\put (320,45){\vector(0,-1){33}}

\end{picture}

\begin{picture}(400,80)
\put (-4,0){$B(hP_2)$}\put (-4,50){$A(hP_2)$}\put (100,0){$I(B)(hP_2)$}\put (100,50){$I(A)(hP_2)$}
\put (-7,25){$f_{hP_2}$}\put (55,60){$\eta_{A(hP_2)}$}
\put(125,25){$If_{hP_2}$;}\put(52,25){($hD$)}
\put(55,10){$\eta_{B(hP_2)}$}
\put(35,3){\vector(1,0){58}}\put(35,55){\vector(1,0){58}}
\put(15,45){\vector(0,-1){33}}\put (120,45){\vector(0,-1){33}}

\put (190,0){$B(hvP_2)$}\put (190,50){$A(hvP_2)$}\put (300,0){$I(B)(hvP_2)$.}\put (300,50){$I(A)(hvP_2)$}
\put (189,25){$f_{hvP_2}$}\put (255,60){$\eta_{A(hvP_2)}$}
\put(325,25){$If_{hvP_2}$}\put(252,25){($hvD$)}
\put(255,10){$\eta_{B(hvP_2)}$}
\put(235,3){\vector(1,0){58}}\put(235,55){\vector(1,0){58}}
\put(215,45){\vector(0,-1){33}}\put (320,45){\vector(0,-1){33}}

\end{picture}\\

The three first squares $(D_i)$ $(i=0,1,2)$ are pullbacks since $\eta_{A(P_i)}$ and $\eta_{B(P_i)}$ are identity maps for $i=0,1,2$ (cf.\ diagram $(2.1)$ and the definition of the reflection $H\vdash I:2Cat\rightarrow 2Preord$ in $(5.2)$).\\

Notice that if diagram $(2.1)$ is restricted to the (vertical) precategory diagram $vP$, one obtains from $(7.1)$ the following square in $Cat$, with unit morphisms of the reflection of all categories into preorders $Cat\rightarrow Preord$ (cf.\ \cite{X:ml}),

\begin{picture}(100,80)
\put (35,0){$B(vP)$}\put (35,50){$A(vP)$}\put (140,0){$I(B)(vP)$.}\put (140,50){$I(A)(vP)$}
\put (25,25){$f_{vP}$}\put (90,60){$\eta_{A(vP)}$}
\put(160,25){$If_{vP}$}\put(200,25){$(7.2)$}
\put(90,10){$\eta_{B(vP)}$}
\put(70,3){\vector(1,0){58}}\put(70,55){\vector(1,0){58}}
\put(50,45){\vector(0,-1){33}}\put (155,45){\vector(0,-1){33}}
\end{picture}\\

It is known (cf.\ \cite[Proposition 3.1]{X:ml}) that this square is a pullback in $Cat$ if and only if, for every two objects $h$ and $h'$ in $A(P_1)$ with $Hom_{A(2P_1)}(h,h')$ nonempty, the map
$$Hom_{A(2P_1)}(h,h')\rightarrow Hom_{B(2P_1)}(f_1h,f_1h')$$
induced by $f$ is a bijection.

A necessary condition for the 2-functor $f$ to be a trivial covering was just stated; the following Lemma \ref{lemma:remaining pullbacks} will help to show that this necessary condition is also sufficient in next Proposition \ref{proposition:trivial coverings}.

\begin{lemma}\label{lemma:remaining pullbacks}
Consider the following commutative parallelepiped

\setlength{\unitlength}{0.7mm}
\begin{picture}(100,90)(50,0)

\put(60,0){$B_2$}\put(70,-5){\vector(1,-1){27}}\put(70,-20){$\eta_{B,2}$}
\put(75,5){\vector(1,0){40}}\put(75,0){\vector(1,0){40}}
\put(90,-10){$r^B$}\put(90,10){$q^B$}
\put(120,0){$B_1$}\put(128,-5){\vector(2,-1){48}}\put(162,-20){$\eta_{B,1}$}
\put(135,5){\vector(1,0){40}}\put(135,0){\vector(1,0){40}}
\put(157,-10){$d^B$}\put(157,10){$c^B$}
\put(180,0){$B_0$}\put(190,-2){\vector(2,-1){53}}\put(215,-10){$\eta_{B,0}$}

\put (55,55){$f_2$} \put (128,55){$f_1$}\put
(188,55){$f_0$}\put (65,65){\vector(0,-1){53}}\put
(183,65){\vector(0,-1){53}}\put (123,65){\vector(0,-1){53}}

\put(60,70){$A_2$}\put(70,65){\vector(1,-1){25}}\put(76,45){$\eta_{A,2}$}
\put(75,75){\vector(1,0){40}}\put(75,70){\vector(1,0){40}}
\put(90,60){$r^A$}\put(90,80){$q^A$}
\put(120,70){$A_1$}\put(135,65){\vector(1,-1){25}}\put(162,45){$\eta_{A,1}$}
\put(135,75){\vector(1,0){40}}\put(135,70){\vector(1,0){40}}
\put(157,60){$d^A$}\put(157,80){$c^A$}
\put(180,70){$A_0$}\put(190,65){\vector(2,-1){50}}\put(215,60){$\eta_{A,0}$}

\put(93,-40){$I(B)_2$}
\put(115,-35){\vector(1,0){50}}\put(115,-40){\vector(1,0){50}}
\put(130,-48){$Ir^B$}\put(130,-30){$Iq^B$}
\put(170,-40){$I(B)_1$}
\put(190,-35){\vector(1,0){50}}\put(190,-40){\vector(1,0){50}}
\put(205,-48){$Id^B$}\put(205,-30){$Ic^B$}
\put(245,-40){$I(B)_0$\hspace{10pt,}}

\put(106,-15){$If_2$} \put (180,-15){$If_1$}\put
(250,-15){$If_0$}\put (105,25){\vector(0,-1){53}}\put
(248,25){\vector(0,-1){53}}\put (177,25){\vector(0,-1){53}}

\put(95,30){$I(A)_2$}
\put(115,30){\vector(1,0){50}}\put(115,35){\vector(1,0){50}}
\put(130,23){$Ir^A$}\put(130,40){$Iq^A$}
\put(168,30){$I(A)_1$}
\put(190,35){\vector(1,0){50}}\put(190,30){\vector(1,0){50}}
\put(205,23){$Id^A$}\put(205,40){$Ic^A$}
\put(245,30){$I(A)_0$}\put(260,10){$(7.3)$}
\end{picture}

\vspace{120pt}

\noindent where the five squares $c^Aq^A=d^Ar^A$, $c^Bq^B=d^Br^B$, $Ic^AIq^A=Id^AIr^A$, $If_0\eta_{A,0}=\eta_{B,0}f_0$ and $If_1\eta_{A,1}=\eta_{B,1}f_1$ are pullbacks.

Then, the square $If_2\eta_{A,2}=\eta_{B,2}f_2$ is also a pullback.\footnote{The notation used in diagram $(7.3)$ is arbitrary, being so chosen in order to make the application of Lemma \ref{lemma:remaining pullbacks} in this section more easily understandable.}

\end{lemma}

\begin{proof} The proof is obtained by an obvious diagram chase.
\end{proof}

\begin{proposition}\label{proposition:trivial coverings}
A 2-functor $f:A\rightarrow B$ is a trivial covering for the reflection $H\vdash I:2Cat\rightarrow 2Preord$ (in notation, $f\in \mathcal{M}_I$) if and only if, for every two objects $h$ and $h'$ in $A(P_1)$ with $Hom_{A(2P_1)}(h,h')$ nonempty, the map
$$Hom_{A(2P_1)}(h,h')\rightarrow Hom_{B(2P_1)}(f_1h,f_1h')$$
induced by $f$ is a bijection.
\end{proposition}

\begin{proof}
In the considerations just above, it was showed that the statement warrants that the squares $(2D)$ and $(vD)$ are pullbacks, adding to the fact that $(D_0)$, $(D_1)$ and $(D_2)$ are all three pullbacks.

Then, $(hD)$ and $(hvD)$ must also be pullbacks according to Lemma \ref{lemma:remaining pullbacks}.
\end{proof}

\section{Coverings for 2-categories via 2-preorders}
\label{sec:Coverings}

A 2-functor $f:A\rightarrow B$ belongs to the class $\mathcal{M}^*_I$ of coverings (with respect to the reflection $H\vdash I:2Cat\rightarrow 2Preord$) if there is some effective descent morphism (also called monadic extension in categorical Galois theory) $p:C\rightarrow B$ in $2Cat$ with codomain $B$ such that the pullback $p^*(f):C\times_BA\rightarrow C$ of $f$ along $p$ is a trivial covering ($p^*(f)\in\mathcal{M}_I$).\\

The following Lemma \ref{lemma:sufficient condition for being covering} can be found in \cite[Lemma 4.2]{X:ml}, in the context of the reflection of categories into preorders, but for 2-categories via 2-preorders the proof is exactly the same, since the same conditions hold (cf.\ Theorem \ref{theorem:stable units} and Example \ref{example:EDM(2Cat)}). The next Proposition \ref{proposition:coverings} characterizes the coverings for 2-categories via 2-preorders.

\begin{lemma}\label{lemma:sufficient condition for being covering}

A 2-functor $f:A\rightarrow B$ in $2Cat$ is a covering (for the reflection $H\vdash I:2Cat\rightarrow 2Preord$) if and only if, for every 2-functor $\varphi :X\rightarrow B$ over $B$ from any 2-preorder $X$, the pullback $X\times_BA$ of $f$ along $\varphi$ is also a 2-preorder.

\end{lemma}

\begin{proposition}\label{proposition:coverings}

A 2-functor $f:A\rightarrow B$ in $2Cat$ is a covering (for the reflection $H\vdash I:2Cat\rightarrow 2Preord$) if and only if it is faithful vertically with respect to 2-cells, that is, for every pair of morphisms $g$ and $g'$, the map $$Hom_{A(2P_1)}(g,g')\rightarrow Hom_{B(2P_1)}(f_1g,f_1g')$$ induced by $f$ is an injection.
\end{proposition}

\begin{proof}
Consider again the 2-preorder $T$ generated by the diagram \begin{picture}(60,40)(0,0)

\put (0,0){$a$}\put (50,0){$a'$.}

\put(10,12){\vector(1,0){35}}\put(25,16){$h$}
\put(10,-6){\vector(1,0){35}}\put(25,-16){$h'$}

\put(23,0){$\Downarrow$}\put(32,0){$\leq$}

\end{picture}\\

If $f$ is not faithful vertically with respect to 2-cells, then, by including $T$ in $B$, one could obtain a pullback $T\times_BA$ that is not a preorder.

Therefore, $f$ is not a covering, by the previous Lemma \ref{lemma:sufficient condition for being covering}.\\

Reciprocally, consider any 2-functor $\varphi :X\rightarrow B$ such that $X$ is a 2-preorder.

If $f$ is faithful (vertically with respect to 2-cells), then the pullback $X\times_BA$ is a 2-preorder, given the nature of $X$. Hence, $f$ is a covering, by the previous Lemma \ref{lemma:sufficient condition for being covering}.

\end{proof}

\section*{Acknowledgement}
This work was supported by
 The Center for Research and Development in Mathematics and Applications (CIDMA) through the Portuguese Foundation for Science and Technology

(FCT - Funda\c{c}\~ao para a Ci\^encia e a Tecnologia),

references UIDB/04106/2020 and UIDP/04106/2020.


\begin{thebibliography}{99}


\bibitem{CJKP:stab}Carboni, A., Janelidze, G., Kelly, G. M., Par\'{e}, R. \textit{On localization and stabilization for factorization systems.} App. Cat. Struct. \textbf{5}, (1997) 1--58.

\bibitem{CHK:fact}Cassidy, C., H\'{e}bert, M., Kelly, G. M. \textit{Reflective subcategories, localizations and factorization systems.} J. Austral. Math. Soc. \textbf{38A} (1985) 287--329.

\bibitem{Eil:trans}Eilenberg, S. \textit{Sur les transformations continues d'espaces m\'{e}triques compacts.} Fundam. Math. \textbf{22} (1934) 292--296.

\bibitem{def_fact_sys}Freyd, P. J., Kelly, G. M. \textit{Categories of continuous funtors I}, J. Pure App. Algebra
\textbf{2} (1972) 169--191.


\bibitem{G. Janelidze}Janelidze, G. \textit{Pure Galois theory in categories}, J. Algebra \textbf{132} (1990) 270--286.

\bibitem{JST:edm}Janelidze, G., Sobral, M., Tholen, W. \textit{Beyond Barr Exactness: Effective Descent
Morphisms} in Categorical Foundations. Special Topics in Order,
Topology, Algebra and Sheaf Theory, Cambridge University Press,
2004.

\bibitem{Sgr->Slat}Janelidze, G., Laan, V., M\'{a}rki, L. \textit{Limit preservation properties of the greatest
semilattice image functor}, Internat. J. Algebra Comput. \textbf{18(5)} (2008), 853--867.

\bibitem{for:ins:sep:factorization}Janelidze, G., Tholen, W. \textit{Functorial factorization,
well-pointedness and separability}, J. Pure App. Algebra
\textbf{142} (1999) 99--130.

\bibitem{SM:cat}Mac Lane, S. \textit{Categories for the Working Mathematician}, 2nd ed., Springer, 1998.

\bibitem{IsabelPhD}Xarez, I. A. \textit{Reflections of Universal Algebras into Semilattices, their Galois Theories and Related Factorization Systems}, University of Aveiro, Ph.D. Thesis, 2013.

\bibitem{X:ml}Xarez, J. J. The monotone-light factorization for categories via
preorders. \textit{Galois theory, Hopf algebras and semiabelian
Categories}, 533--541, Fields Inst. Commun., 43, Amer. Math. Soc.,
Providence, RI, 2004.

\bibitem{X:iml}Xarez, J. J. \textit{Internal monotone-light factorization for categories via preorders}, Theory
Appl. Categories 13 (2004) 235--251.

\bibitem{X:GenConComp}Xarez, J. J. \textit{Generalising Connected Components},
J. Pure Appl. Algebra, 216, Issues 8-9(2012), 1823--1826.

\bibitem{Why:trans}Whyburn, G. T. \textit{Non-alternating transformations.} Amer. J. Math. \textbf{56} (1934) 294--302.


\end{thebibliography}
\end{document}